\documentclass[11pt]{amsart}
\usepackage{amsmath, amssymb, amsthm, mathrsfs}
\usepackage{geometry}
\newcommand{\amsprimary}[1]{{\footnotesize\noindent AMS 2020 \textit{Mathematics subject
classification:} Primary #1\vspace{1pc}}}
\newcommand{\keywordsnames}[1]{{\footnotesize\noindent\textit{Key words:} #1\vspace{1pc}}}

\title[Maximum Principle, Eigenvalues and Stabilization]{A maximum principle for the $p$-Laplacian, an eigenvalue estimate
and a stabilization phenomenon for the large-$p$ regime}
\author{Kevin Carrillo-Reina and Jean C. Cortissoz}
\address{Universidad de los Andes, Bogot\'a DC, COLOMBIA}
\email{k.carrillor@uniandes.edu.co, jcortiss@uniandes.edu.co}
\date{}

\newtheorem{theorem}{Theorem}[section]
\newtheorem{lemma}[theorem]{Lemma}

\newtheorem{corollary}[theorem]{Corollary}
\theoremstyle{definition}

\newtheorem{remark}[theorem]{Remark}

\begin{document}

\begin{abstract}
We establish an explicit maximum principle for the Dirichlet problem associated with the $p$-Laplacian ($p>1$), where the constant depends on both $p$ and the geometry of the domain. From this result we derive two main applications. First, we obtain a new lower bound for the first nontrivial eigenvalue of the $p$-Laplacian, which improves upon existing estimates in certain parameter regimes and for thin domains. Second, we prove an existence theorem for nonlinear boundary value problems of the form
\[
-\Delta_p u = \lambda f(u) \quad \text{in } \Omega, \qquad u=0 \quad \text{on } \partial \Omega,
\]
with $f$ nonnegative, continuous and nondecreasing. A striking consequence is the emergence of a \emph{stabilization phenomenon}: for every such nonlinearity there exists a threshold $p_0 \colon = p_0(f,\lambda,\Omega)$ such that for all $p \geq p_0$ 
solutions exist. To our knowledge, this stabilization effect with respect to $p$, that apparently has not been observed before, suggests a connection to the $\infty$-Laplacian.
\end{abstract}

\maketitle

\keywordsnames{p-Laplacian; maximum principle; eigenvalue estimate; nonlinear boundary value problems; stabilization phenomenon; infinity Laplacian}

{\amsprimary {35J92, 35B50, 35J70, 35J60}}

\section{Introduction: An Extended Maximum Principle}

The $p$-Laplacian,
\[
\Delta_p u = \operatorname{div}\!\left( |\nabla u|^{p-2} \nabla u \right),
\]
arises naturally as the Euler--Lagrange operator of the variational functional
\[
u \mapsto \int_\Omega |\nabla u|^p \, dV,
\]
and plays a central role in nonlinear potential theory, variational problems, and the study of degenerate and singular elliptic equations. While the classical Laplacian ($p=2$) enjoys well-established maximum and comparison principles \cite{GilbargTrudinger}, the case $p \neq 2$ is more delicate due to the operator's singularity when $1<p<2$ and degeneracy when $p>2$.

\medskip
The goal of this paper is twofold. We first establish a \emph{quantitative maximum principle} for the Dirichlet problem associated with the $p$-Laplacian, where the bound is given explicitly in terms of $p$ and the slab diameter of the domain. Comparison principles for the $p$-Laplacian are known \cite{Tolksdorf1983,GarciaMelianSabina1998};
however, explicit constants of the type proved in this paper do not appear to have been recorded in the literature.

\medskip
Building on this principle, we obtain two applications:
\begin{enumerate}
    \item \textbf{Eigenvalue estimate.} We derive a new lower bound for the first nontrivial Dirichlet eigenvalue of the $p$-Laplacian. Our bound improves upon existing results in certain parameter ranges, particularly when $1<p<2$ and for thin domains. This contributes to the ongoing effort to sharpen eigenvalue estimates for nonlinear elliptic operators \cite{Aghajani2017, BenediktDrabek2012,  KawohlFridman2003}.
    \item \textbf{Stabilization phenomenon.} We prove an existence result for nonlinear boundary value problems of the form
    \[
    -\Delta_p u = \lambda f(u) \quad \text{in } \Omega, \qquad u=0 \quad \text{on } \partial \Omega,
    \]
    where $f$ is positive, continuous and nondecreasing. A remarkable outcome is the discovery of a \emph{stabilization effect}: for any $\lambda>0$ and such nonlinearity $f$, there exists a threshold $p_0\colon =p_0\left(f,\lambda,\Omega\right)$ such that for all $p \geq p_0$ the problem admits a positive solution. To the best of our knowledge, this type of stabilization with respect to the index $p$ has not been reported previously.
\end{enumerate}

This latter phenomenon is of particular interest, as it hints at connections with the theory of the $\infty$-Laplacian.

\medskip
A key geometric quantity in our maximum principle is the slab diameter of the domain, which we now define.
Given a smooth bounded open set, we define its slab diameter as
the infimum of the distance between two parallel hyperplanes such that $\Omega$
is contained in the region bounded by them. 
 
\medskip
Without further ado, we state and prove our 
extended maximum principle for the $p$-Laplacian.

\begin{theorem}[Extended Maximum Principle for the $p$-Laplacian]
\label{thm:maximum_principle}
Let $\Omega \subset \mathbb{R}^n$ be a smooth bounded open set with slab diameter bounded above by $d$, and let $h \in L^\infty(\Omega)$. Then for any $p>1$ the Dirichlet problem
\[
\begin{cases}
-\Delta_p u = h & \text{in } \Omega, \\
u = 0 & \text{on } \partial\Omega,
\end{cases}
\]
admits a unique weak solution $u \in W^{1,p}_0(\Omega) \cap L^\infty(\Omega)$. Moreover, 
\[
\|u\|_{\infty} \leq \left(\frac{p-1}{p}\right)
\left(\frac{d}{2}\right)^{\frac{p}{p-1}} \|h\|_{\infty}^{\frac{1}{p-1}}.
\]
\medskip
Recall that $u\in W^{1,p}_0\left(\Omega\right)$ is a weak solution to the BVP above if
for all $\varphi\in C_c^{\infty}\left(\Omega\right)$ 
\[
\int_{\Omega} \left|\nabla u\right|^{p-2}\nabla u\nabla \varphi \,dV= \int_{\Omega}h\left(x\right)\varphi\,dV
\]
holds.
\end{theorem}
\begin{proof}
The existence part of the statement is a well known consequence of the Minty-Browder theorem (see Theorem 5.16 in \cite{BrezisBook}), so let us prove the $L^{\infty}$-norm estimate.
Without loss of generality, we may assume that
$\Omega$ is contained in the slab $\mathbb{R}^{n-1}\times \left(-\frac{d}{2},\frac{d}{2}\right)$.
Consider the function
\[
\psi\left(x\right)= \left(\frac{p-1}{p}\right)\left[\left(\frac{d}{2}\right)^{\frac{p}{p-1}}-
\left|x_n\right|^{\frac{p}{p-1}}\right]\left\|h\right\|_{\infty}^{\frac{1}{p-1}}.
\]
Notice that $\psi$ is $C^1$. 
Let \(\alpha = \frac{p}{p-1} > 1\) and define
\[
u(x) = |x_n|^{\alpha}.
\]
Since \(u\) depends only on \(x_n\),
\[
\nabla u = \alpha |x_n|^{\alpha-1} \operatorname{sgn}(x_n) \, e_n,
\qquad
|\nabla u| = \alpha |x_n|^{\alpha-1},
\]
where $e_n=(0,\ldots, 0,1)$. Thus,
\[
|\nabla u|^{p-2} \nabla u
= \left( \alpha |x_n|^{\alpha-1} \right)^{p-2}
\left( \alpha |x_n|^{\alpha-1} \operatorname{sgn}(x_n) \right) e_n
= \alpha^{p-1} |x_n|^{(\alpha-1)(p-1)} \operatorname{sgn}(x_n) \, e_n.
\]
Since \((\alpha - 1)(p-1) = 1\), we get
\[
|\nabla u|^{p-2} \nabla u = \alpha^{p-1} x_n \, e_n.
\]
Taking the divergence,
\[
\Delta_p u = \operatorname{div} \left( \alpha^{p-1} x_n \, e_n \right)
= \alpha^{p-1}.
\]
Therefore,
\[
- \Delta_p \left( |x_n|^{\frac{p}{p-1}} \right)
= - \left( \frac{p}{p-1} \right)^{p-1}.
\]
Hence,
we have that,
\[
-\Delta_p \psi = \left\|h\right\|_{\infty},
\]
in the weak sense.
Therefore, by the Comparison Principle for the $p$-Laplacian
(see Lemma 3.1 in  \cite{Tolksdorf1983} or Section 5 in \cite{GarciaMelianSabina1998})
\[
-\psi \leq u \leq \psi.
\]
Thus, 
\[
\left\|u\right\|_{\infty}\leq \left\|\psi\right\|_{\infty}=\psi\left(0\right)
=\left(\frac{p-1}{p}\right)
\left(\frac{d}{2}\right)^{\frac{p}{p-1}} \|h\|_{\infty}^{\frac{1}{p-1}}.
\]
\end{proof}

In the sections that follow we present the applications
described above starting with an estimate for
the first eigenvalue of the $p$-Laplacian (Section \ref{section:eigenvalue_estimate}),
and an existence theorem for a BVP with the stabilization
phenomenon as a striking consequence (Section \ref{section:existence_result}). We finish
the paper (Section \ref{section:final_digression})
with a digression on a possible connection with solutions to the infinity Laplacian:
the solutions given by our existence theorem form a compact set in some H\"older space
as $p\rightarrow \infty$ which raises the question of a subsequence of solutions converging 
towards a possible viscosity solution to the BVP problem stated above with $p=\infty$.

\section{Application: An Eigenvalue Estimate}
\label{section:eigenvalue_estimate}

Having established our maximum principle, we now apply it to derive a new lower bound for the first eigenvalue of the $p$-Laplacian, which 
is defined as 
\[
\lambda_{1,p}\left(\Omega\right)
=\inf_{v\in W_0^{1,p}\left(\Omega\right)}
\frac{\int_{\Omega}\left|\nabla v\right|^p\, dV}{\int_{\Omega}
\left|u\right|^p \,dV}.
\]
We will give a new estimate from below for this eigenvalue, 
which improves on some known estimates at least in certain regimes.
Our main tool
is the Extended Maximum Principle. Here is our estimate.
\begin{corollary}
    Let $\Omega\subset \mathbb{R}^n$ be a smooth open subset with slab diameter $\leq d$.
    Then, its first $p$-Laplacian eigenvalue satisfies
    \[
    \lambda_{1,p}\left(\Omega\right)\geq \left(\frac{p}{p-1}\right)^{p-1}\left(\frac{2}{d}\right)^{p}. 
    \]
\end{corollary}
\begin{proof}
    Let $u$ be a first eigenfunction. Then, it satisfies
    \[
    \Delta_p u = \lambda_{1,p}\left(\Omega\right)\left|u\right|^{p-2}u,
    \]
    and hence, by the Extended Maximum Principle,
    \[
    \left\|u\right\|_{\infty}\leq \left(\frac{p-1}{p}\right)
\left(\frac{d}{2}\right)^{\frac{p}{p-1}} 
\left(\lambda_{1,p}\left(\Omega\right)\left\|u\right\|_{\infty}^{p-1}\right)^{\frac{1}{p-1}},
    \]
    and hence,
    \[
    1\leq \left(\frac{p-1}{p}\right)
\left(\frac{d}{2}\right)^{\frac{p}{p-1}}\lambda_{1,p}\left(\Omega\right)^{\frac{1}{p-1}}.
    \]
    Solving for $\lambda_{1,p}\left(\Omega\right)$ gives the estimate.
\end{proof}

\subsection{Comparison with known results}
\label{section:known_results}
Very few results exist regarding lower bounds for the first eigenvalue for the $p$-Laplacian.
In the ball $B_{R}\left(0\right)\subset \mathbb{R}^n$, it is known that
(see \cite{BenediktDrabek2013})
\[
\lambda_{1,p}\left(B_R\right)\geq \frac{n}{R^p}\left(\frac{p}{p-1}\right)^{p-1}.
\]
Notice the similarity with our estimate. For domains contained in a slab
$\mathbb{R}^{n-1}\times \left(-\frac{d}{2},\frac{d}{2}\right)$, Benedikt and Dr\'abek
showed in \cite{BenediktDrabek2012} the estimate
\[
\lambda_{1,p}\left(\Omega\right)\geq \frac{2^p p}{d^p}.
\]
Notice that their estimate is better than ours for large $p$; however, for
$1<p<2$ ours is better. For instance, for $p=\frac{3}{2}$ we get
\[
\lambda_{1,p}\left(\Omega\right)\geq \frac{2^{\frac{3}{2}}\sqrt{3}}{d^{\frac{3}{2}}},
\]
which is better than
\[
\lambda_{1,p}\left(\Omega\right)\geq \frac{2^{\frac{3}{2}}\times 1.5}{d^{\frac{3}{2}}}.
\]
On the other hand, Aghajani and Tehrani gave the following estimate in terms of the
diameter of $\Omega$, extending those known for the ball:
\[
\lambda_{1,p}\left(\Omega\right)\geq \frac{2^p n}{\mbox{diam}\left(\Omega\right)^p}\left(\frac{p}{p-1}\right)^{p-1}.
\]
Clearly, our estimate is better in the case of a domain that
is very thin in one direction but very long in another.

\medskip
Lindqvist showed that
\[
\lambda_{1,p}\left(\Omega\right)\geq \left(\frac{2}{p}\right)^p\lambda_1\left(\Omega\right)^{\frac{p}{2}},
\]
which, given our best estimates for the first eigenvalue of the Laplacian in a convex domain gives
\[
\lambda_{1,p}\left(\Omega\right)\geq \left(\frac{2\pi}{p d}\right)^p,
\]
where $d$ is the diameter of the domain. Again, the estimate given
in this paper is better for large $p$.

\begin{remark}[Comparison with Kawohl–Fridman]
The classical estimate by Kawohl and Fridman \cite{KawohlFridman2003} asserts that for a domain $\Omega$ of diameter $d$, the first Dirichlet eigenvalue of the $p$-Laplacian satisfies
\[
\lambda_{1,p}(\Omega) \geq \left( \frac{h\left(\Omega\right)}{p} \right)^p,
\]
where $h\left(\Omega\right)$ is Cheeger's constant.
In the case of a ball of radius $R$, Kawohl-Fridman gives the estimate 
\[
\lambda_{1,p}(B_R\left(0\right))\geq \left(\frac{n}{p}\right)^p\left(\frac{1}{R}\right)^p,
\]
and the constant $\left(n/p\right)^p\rightarrow 0$ as $p\rightarrow \infty$
 In contrast, our explicit estimate
\[
\lambda_{1,p}(B_R\left(0\right)) \geq \left( \frac{p}{p-1} \right)^{p-1} \left( \frac{1}{R} \right)^p
\]
is sharper for large $p$, since $\left(\dfrac{p}{p-1}\right)^{p-1}\rightarrow e$.


\end{remark}

\section{Application: An Existence Theorem and the Stabilization Phenomenon}
\label{section:existence_result}

As our second application of Theorem \ref{thm:maximum_principle}, we show an existence result for a nonlinear Boundary Value Problem for the $p$-Laplacian,
namely the following theorem, for which we recall that a function 
    $u\in W^{1,p}_0\left(\Omega\right)\cap L^{\infty}\left(\Omega\right)$ is a weak solution to the boundary value problem 
    \begin{equation}
    \label{eq:BVP}
        \left\{
        \begin{array}{l}
        -\Delta_p u = \lambda f\left(u\right), \quad \mbox{in} \quad \Omega,\\
        u=0 \quad \mbox{on} \quad \partial \Omega,
        \end{array}
        \right.
    \end{equation}
    if for every $\varphi\in C_c^\infty(\Omega)$ we have
    \[
\int_{\Omega} \left|\nabla u\right|^{p-2}\nabla u\cdot\nabla \varphi \,dx= \lambda\int_{\Omega}f\left(u\right)\varphi\,dx.
\]
\begin{theorem}
\label{thm:existence}
    Consider the BVP problem \eqref{eq:BVP}
    where $f$ is a nondecreasing, continuous, and nonnegative function. 
    If the map 
    \begin{equation}
    \label{eq:fixed_point_map}
    x\mapsto \lambda^{\frac{1}{p-1}}\left(\frac{p-1}{p}\right)
\left(\frac{d}{2}\right)^{\frac{p}{p-1}}f\left(x\right)^{\frac{1}{p-1}}
     \end{equation}
    has a positive fixed point $x_0$, and (\ref{eq:BVP}) has a nonnegative weak subsolution $u_0$ such that
    $\left\|u_0\right\|_{\infty}<x_0$, and either $f(0)>0$, or $f(0)=0$ and $u_0$ is nontrivial, then \eqref{eq:BVP} has a weak solution $u$
    such that $\left\|u\right\|_{\infty}\leq x_0$.
\end{theorem}

Before we give a proof of the previous theorem, let us turn to an example.
Consider the BVP
\begin{equation}
    \label{eq:BVP1}
        \left\{
        \begin{array}{l}
        -\Delta_p u = \lambda e^{u^q}, \quad \mbox{in} \quad \Omega,\\
        u=0 \quad \mbox{on} \quad \partial \Omega,
        \end{array}
        \right.
    \end{equation}
    with $q>0$. We observe that $u_0=0$ is a subsolution of (\ref{eq:BVP1}) satisfying the hypotheses in Theorem \ref{thm:existence} if $\lambda>0$ or
    $d>0$ is small enough, or if $p>1$ is large enough, since the map 
    \[
    x\mapsto \lambda^{\frac{1}{p-1}} \left(\frac{p-1}{p}\right)
\left(\frac{d}{2}\right)^{\frac{p}{p-1}}e^{\frac{x^q}{p-1}}
    \]
    has a positive fixed point, and hence the BVP has at least a positive solution.
\begin{proof}
    Now we give a proof of Theorem \ref{thm:existence}. The idea is that the maximum principle produces a uniform bound to the iterative scheme, which we now describe. Indeed,  from the subsolution $u_0$, we define $u_{m+1}$, for $m=0,1,2,3,\dots$, as the solution of the boundary
    value problem
    \[
     -\Delta_p u_{m+1}=\lambda f\left(u_m\right),
     \quad \mbox{in} \quad \Omega, \quad u_{m+1}=0
     \quad \mbox{on}\quad \partial\Omega.
    \]
    From the monotonicity of $f$ and the comparison principle it
    follows that $u_{m}\leq u_{m+1}$. Next, we prove 
    that if $\left\|u_m\right\|_{L^{\infty}\left(\Omega\right)}\leq x_0$ then $\left\|u_{m+1}\right\|_{L^{\infty}\left(\Omega\right)}\leq x_0$.
    In fact,
    \begin{eqnarray*}
        \left\|u_{m+1}\right\|_{\infty}&\leq& 
     \lambda^{\frac{1}{p-1}} \left(\frac{p-1}{p}\right)
\left(\frac{d}{2}\right)^{\frac{p}{p-1}} f\left(\left\|u_m\right\|_{\infty}\right)^{\frac{1}{p-1}}\\
        &\leq&\lambda^{\frac{1}{p-1}}\left(\frac{p-1}{p}\right)
\left(\frac{d}{2}\right)^{\frac{p}{p-1}} f\left(x_0\right)^{\frac{1}{p-1}}=x_0.
    \end{eqnarray*}
    From this, the sequence $(u_m)$ is pointwise convergent to some function $u$ such that $\Vert u\Vert_\infty\leq x_0$. So, as shown in Section \ref{sect:monotone_iteration}, the method of monotone iteration is employed to prove that $u$ is a solution of (\ref{eq:BVP}).
\end{proof}

An unexpected consequence of the theorem above is the following
stabilization result.
\begin{corollary}
\label{cor:existence}
    Let $f:\left[0,\infty\right)\longrightarrow \left(0,\infty\right)$ be a continuous nondecreasing function.
    There exists a $p_0\colon = p_0\left(f,\lambda,\Omega\right)$ such that
    if $p\geq p_0$ then (\ref{eq:BVP}) has at least 
    one positive solution.
\end{corollary}
\begin{proof}
    The function $u_0=0$ is a subsolution of (\ref{eq:BVP}) under
    the hypotheses on $f$. Note that 
    \[
    f\left(x\right)^{\frac{1}{p-1}}\rightarrow 1 \quad 
    \mbox{as}\quad p\rightarrow\infty
    \]
    uniformly on compact sets, 
    and hence, by the intermediate value theorem, that for $p$ large enough the function
    \[
    x\mapsto \lambda^{\frac{1}{p-1}}  \left(\frac{p-1}{p}\right)
\left(\frac{d}{2}\right)^{\frac{p}{p-1}}f\left(x\right)^{\frac{1}{p-1}}
    \]
    has at least a positive fixed point. The result follows from
    Theorem \ref{thm:existence}.
\end{proof}

The value of $p_0$ can be estimated explicitly. Let us show how.
Pick $P_1$ and $P_2$ such that 
\begin{itemize}
\item for all $p\geq P_1$
\[
\lambda^{\frac{1}{p-1}}\left(\frac{p-1}{p}\right)
\left(\frac{d}{2}\right)^{\frac{p}{p-1}}f\left(2\lambda d/2\right)^{\frac{1}{p-1}}\leq \frac{3}{2}\lambda 
\frac{d}{2},
\]
\item for all $p\geq P_2$
\[
\lambda^{\frac{1}{p-1}}\left(\frac{p-1}{p}\right)
\left(\frac{d}{2}\right)^{\frac{p}{p-1}} f\left(0\right)^{\frac{1}{p-1}}\leq \frac{5}{4}\lambda \frac{d}{2}.
\]
\end{itemize}
We then let $p_0=\max\left\{P_1,P_2\right\}$. With this choice, an application of the Banach's fixed point shows that the map
(\ref{eq:fixed_point_map}) has a fixed point in $\left(0,2\lambda d/2\right)$. 

\medskip
Let us put some numbers in the previous estimate with a concrete example. Consider the BVP
\[
-\Delta_p u = e^{u^3} \quad \mbox{in} \quad B_1\left(0\right), \quad u=0 \quad \mbox{on}
\quad \partial B_1\left(0\right),
\]
where $B_1\left(0\right)$ is the ball of radius 1 centered at the origin (so here $d=2$). In this case, the previous conditions
translate into finding $P_1$ and $P_2$ such that
\begin{itemize}
    \item for all $p\geq P_1$
    \[
\left(\frac{p-1}{p}\right)
\exp\left(\frac{2^3}{p-1}\right)\leq \frac{3}{2},
\]
\item for all $p\geq P_2$
\[
\left(\frac{p-1}{p}\right)
 \leq \frac{5}{4},
\]
\end{itemize}
which is easily seen to be satisfied for $P_1,P_2=\frac{2^3}{\ln(3/2)}+1$, and thus $p_0=\frac{2^3}{\ln(3/2)}+1$, which says that the BVP has a solution 
for all $p\geq \frac{2^3}{\ln(3/2)}+1$, since the map $p\mapsto\left(\frac{p-1}{p}\right)\exp\left(\frac{2^3}{p-1}\right)$ is strictly decreasing with respect to $p$ in $(1,\infty)$.

\subsection{The method of monotone iteration}
\label{sect:monotone_iteration}
In this subsection we finish the proof of Theorem \ref{thm:existence} by showing that the monotone iteration scheme used
converges towards a weak solution
of (\ref{eq:BVP}). Since this proof has several technical details, we divide it into parts. Before, we define some useful notation. By $\ell_m$, with $m\geq 0$, and $\ell$, we denote the bounded linear functionals defined by
$$\ell_m(v):=\int_\Omega\vert\nabla u_{m+1}\vert^{p-2}\nabla u_{m+1}\cdot\nabla v\, dx,\quad\text{ for all }v\in W_0^{1,p}(\Omega),\quad\text{and}$$
$$\ell(v):=\int_\Omega\vert\nabla u_{}\vert^{p-2}\nabla u_{}\cdot\nabla v\, dx,\quad\text{ for all }v\in W_0^{1,p}(\Omega),$$
respectively. Furthermore, we observe that $\ell_m(u_{m+1})=\Vert\nabla u_{m+1}\Vert_p^p$, for each $m\geq 0$.
\medskip
\begin{lemma}\label{lem:mon_it_1}
    The sequence $(u_m)_m$ is bounded in norm $W_0^{1,p}(\Omega)$.
\end{lemma}

\begin{proof} 
Since $f\geq 0$ is a nondecreasing function, the fact that $\Vert u_m\Vert_\infty\leq x_0$, and that $u_m$ satisfies 
\begin{equation}\label{eq:WeakForm_un}
    \int_{\Omega}\vert \nabla u_{n+1}\vert^{p-2}\nabla u_{n+1}\cdot\nabla\varphi\, dx=\int_{\Omega}f(u_n)\varphi\, dx,\quad\text{ for each }\varphi\in C_0^\infty(\Omega),
\end{equation}
for each $m\geq 1$, imply
$$\Vert\nabla u_{m+1}\Vert_p^p=\int_\Omega f(u_m)u_m\, dx\leq\int_\Omega f(x_0)x_0\, dx=\vert\Omega\vert f(x_0)x_0,\quad\text{for every }m\geq 0.$$
\end{proof} 

\begin{lemma}
    There exists a subsequence $(u_{k_m})_m$ such that $\lim_{m\to\infty}u_{k_m}=u$ in norm $W_0^{1,p}(\Omega)$. 
\end{lemma}
\begin{proof}
Since $W_0^{1,p}(\Omega)$ is a reflexive space, Lemma \ref{lem:mon_it_1} implies that there exists $w\in W_0^{1,p}(\Omega)$  such that $u_{k_m}\rightharpoonup w$ in $W_0^{1,p}(\Omega)$ as $m\to\infty$. 
We next prove that $w\equiv u$. Now, since $\Vert u_m-u\Vert_p\to 0$ \textit{as} $m\to\infty$, the weak convergence implies that for each $\varphi\in C_0^\infty(\Omega)$ we have
    $$\int_\Omega u\,\varphi\, dx=\lim_{m\to\infty}\int_\Omega u_{k_m}\varphi\, dx=\int_\Omega w\,\varphi\, dx,$$
    implies that $u\equiv w$ and hence $u\in W_0^{1,p}(\Omega)$, from which follows that $u\equiv w$ in $\Omega$, that is, $u\in W_0^{1,p}(\Omega)$, and $u_{k_m}\rightharpoonup u$ as $m\to\infty$. Now, we prove that $\Vert \nabla u_{k_m}-\nabla u\Vert_p\to 0$ as $m\to\infty.$
    For this purpose, we take into account that $W_0^{1,p}(\Omega)$ is a uniformly convex space, so we only need to prove that
    $\Vert\nabla u_{k_m}\Vert_p\to\Vert\nabla u\Vert_p$ as $m\to\infty.$
    \medskip
    Observe that the weak convergence implies that $\Vert \nabla u\Vert_{p}\leq \liminf_{m}\Vert \nabla u_{k_m}\Vert_p$. So, we now prove the reverse inequality. Since 
    $$\ell(v)=\lim_{m\to\infty}\int_\Omega f(u_m)v\, dx,\quad\text{ for each }v\in W_0^{1,p}(\Omega),$$
    we observe that \eqref{eq:WeakForm_un} implies that
    $$\vert\ell(v)\vert\leq\left(\lim_{m\to\infty}\int_\Omega f(u_m)\frac{u_{m+1}}{\Vert\nabla u_{m+1}\Vert_p}\, dx\right)\Vert v\Vert_{W_{0}^{1,p}(\Omega)},\quad\text{for each }v\in W_0^{1,p}(\Omega).$$
    Once again, by \eqref{eq:WeakForm_un} we conclude that 
    $$\Vert\nabla u_m\Vert_p^p\to\int f(u)u\, dx,\quad\text{as}\quad m\to\infty,$$
    and consequently,
    $$\Bigg\vert\int_\Omega f(u)v\, dx\Bigg\vert\leq\left[\frac{1}{\left(\int_\Omega f(u)u\, dx\right)^{1/p}}\int_\Omega f(u)u\, dx\right]\Vert v\Vert_{W_{0}^{1,p}(\Omega)}=\left(\int_\Omega f(u)u\, dx\right)^{1-\frac{1}{p}}\Vert v\Vert_{W_{0}^{1,p}(\Omega)},$$
    form which, if $v=u/\Vert\nabla u\Vert_p$, we conclude the reverse inequality. Finally, since $W_0^{1,p}(\Omega)$ is a uniformly convex space, the weak convergence, and the fact that $\lim_{m\to\infty}\Vert\nabla u_{k_m}\Vert_p=\Vert\nabla u\Vert_p$, imply that 
    $$\Vert\nabla u_{k_m}-\nabla u\Vert_p\to 0,\quad\text{as }m\to\infty.$$
\end{proof}
    At this stage, we finish the proof of Theorem \ref{thm:existence}. We observe that there exists a subsequence $(u_{k_m})$ such that
    $$\int_\Omega\vert\nabla u_{k_m}\vert^{p-2}\nabla u_{k_m}\cdot\nabla\varphi\, dx\to\int_\Omega\vert\nabla u\,\vert^{p-2}\nabla u\cdot\nabla\varphi\,dx\quad\text{as}\quad m\to\infty.$$
    Indeed, the fact that there exists a subsequence $(u_{k_m})$ such that $\Vert \nabla u_{k_m}-\nabla u\Vert_p\to 0$ as $m\to\infty$, implies the existence of a further subsequence (not relabeled) $(u_{k_m})_m$ whose gradients converge a.e. to $\nabla u$. Now, the fact that 
    $$\left\|\vert\nabla u_{k_m}\vert^{p-2}\nabla u_{k_m}-\vert\nabla u\,\vert^{p-2}\nabla u\right\|_{\frac{p}{p-1}}\to 0\quad\text{as}\quad m\to\infty,$$ 
    follows by considering the fact that if $(v_m)$ is a sequence in $L^q(\Omega)$, with $1<q<\infty$, and $v_m\to v$ a.e, where $v\in L^q(\Omega)$, and $\Vert v_m\Vert_q\to\Vert v\Vert_q$, then $\Vert v_m-v\Vert_q\to0$ as $m\to\infty$ (see p.p. 123 in \cite{BrezisBook}). In fact, if $q=\frac{p}{p-1}$ and $v_m=\vert\nabla u_{k_m}\vert^{p-2}\nabla u_{k_m}$, we have
    $$\vert\nabla u_{k_m}\vert^{p-2}\nabla u_{k_m}\to\vert\nabla u\,\vert^{p-2}\nabla u\quad\text{ a.e.},$$ and that
    $\left\|\nabla u_{k_m}\right\|_p^p\to\Vert\nabla u\,\Vert_p^p$ can be rewritten as
    $$\left\| \vert\nabla u_{k_m}\vert^{p-2}\nabla u_{k_m} \right\|_{\frac{p}{p-1}}^{\frac{p}{p-1}}\to \left\| \vert\nabla u_{}\vert^{p-2}\nabla u_{} \right\|_{\frac{p}{p-1}}^{\frac{p}{p-1}},\quad\text{as }m\to\infty,$$
   However, the limit $\lim_{k\to\infty}\left\|\nabla u_{k_m}\right\|_p=\Vert\nabla u\,\Vert_p$ was previously proved. Hence, by H\"older's inequality, we obtain that
    $$
         \left|\int_\Omega\left[\vert\nabla u_{k_m}\vert^{p-2}\nabla u_{k_m}-\vert\nabla u\,\vert^{p-2}\nabla u\right]\cdot\nabla\varphi\, dx\right|\leq \left\|\vert\nabla u_{k_m}\vert^{p-2}\nabla u_{k_m}-\vert\nabla u\,\vert^{p-2}\nabla u\right\|_{\frac{p}{p-1}}\Vert\nabla\varphi\Vert_{p},
    $$
    and letting $k\to\infty$, we obtain the desired limit. Therefore,
    $$\int_\Omega\vert\nabla u\vert^{p-2}\nabla u\cdot\nabla\varphi\, dx=\int_\Omega f(u)\varphi\, dx,\quad\text{ for each }\varphi\in C_0^\infty(\Omega).$$

\begin{remark}
    From the proof presented in this section, it should be
    clear to the reader that the continuity assumption
    over $f$ in Theorem \ref{thm:existence} can be relaxed. All that is required
    is that $f\in L^{\infty}_{loc}\left(\Omega\right)$ and that
    the following ``continuity from below" is satisfied:
    if $x_n\nearrow x$ then 
    $f\left(x_n\right)\rightarrow f\left(x\right)$.
\end{remark}

  \section{Digression: A possible connection with the Infinity Laplacian}  
\label{section:final_digression}
From all the previous considerations, since the fixed
points of the family of maps
\[
x\mapsto \lambda^{\frac{1}{p-1}}  \left(\frac{p-1}{p}\right)
\left(\frac{d}{2}\right)^{\frac{p}{p-1}}f\left(x\right)^{\frac{1}{p-1}}
\]
are uniformly bounded as $p\rightarrow\infty$, if we denote by $u_p$ the
solution to (\ref{eq:BVP}) for the parameter $p$, then, since
\[
\int_{\Omega}f\left(u_p\right)u_p\, dx
\]
can be uniformly bounded, 
it follows that there is a constant $M$ independent
of $p$ such that
\[
\left\|u_p\right\|_{W^{1,p}_0\left(\Omega\right)}
\leq M,
\]
and hence, by the the usual proof of Morrey's inequality (see Theorem 4 and  Theorem 5 of Chapter 5 in \cite{EvansBook}), there exists a constant $C$ independent of $p$, for every $p$ large enough, such that
$$\Vert u_p\Vert_{C^{0,1-\frac{n}{p}}(\overline\Omega)}\leq C\Vert u_p\Vert_{W^{1,p}(\Omega)}.$$
Consequently, for every $0<\alpha<1$, it follows that for every $p$ sufficiently large holds
$$\Vert u_p\Vert_{C^{0,\alpha}(\overline \Omega)}\leq\text{Diam}(\Omega)\Vert u_p\Vert_{C^{0,1-\frac{n}{p}}(\overline \Omega)}\leq \text{Diam}(\Omega)C\Vert u_p\Vert_{W_0^{1,p}(\Omega)}\leq \text{Diam}(\Omega)CM.$$
Therefore, for every $0<\alpha<1$ there is a constant $M'$ such that
\[
\left\|u_p\right\|_{C^{0,\alpha}(\overline\Omega)}
\leq M',\quad\text{for all }p\text{ large enough.}
\]
Then, via Arzelà-Ascoli's theorem, there is
a sequence of parameters $p_k\rightarrow \infty$ such that
the corresponding sequence of functions $\left(u_{p_k}\right)_k$ converges 
towards a function $v\in C^{0,\alpha}\left(\Omega\right)$. 
This raises the natural question of whether $v$ related
to solutions of an equation involving the Infinity Laplacian.

\end{document}